\documentclass[12pt]{amsart}
\usepackage{amsthm,amsmath,amssymb}
\numberwithin{equation}{section}

\def\ca{{\mathcal A}}
\def\cb{{\mathcal B}}
\def\cc{{\mathcal C}}

\def\cl{{\mathcal L}}

\def\car{{\mathcal R}}
\def\cas{{\mathcal S}}
\def\ct{{\mathcal T}}

\def\bc{{\mathbb C}}
\def\bd{{\mathbb D}}
\def\baf{{\mathbb F}}

\def\bn{{\mathbb N}}

\def\bt{{\mathbb T}}
\def\bz{{\mathbb Z}}

\def\a{\alpha}
\def\eps{\varepsilon}

\def\d{\delta}

\def\j{\iota}

\def\l{\lambda}

\def\x{\xi}

\def\s{\sigma}
    
\def\t{\tau}

\def\f{\varphi}
\def\c{\chi}

\theoremstyle{plain}
\newtheorem{lemma}{Lemma}[section]
\newtheorem{proposition}[lemma]{Proposition}
\newtheorem{theorem}[lemma]{Theorem}

\theoremstyle{definition}

\newtheorem{definition}[lemma]{Definition}

\theoremstyle{remark}

\begin{document}

\title[Unbounded   Expectations]{\textsc{Unbounded  Expectations to some  von Neumann Algebras  }}

\begin{center}
Dedicated to the memory of Uffe V. Haagerup
\end{center}

\author[E.~Christensen]{Erik Christensen}
\address{\hskip-\parindent
Erik Christensen, Mathematics Institute, University of Copenhagen, Copenhagen, Demark.}
\email{echris@math.ku.dk}
\date{\today}

\begin{abstract}
For any injective von Neumann algebra $\car$ on a Hilbert space $H$ and any discrete, countable group $G,$ which acts by *-automorphisms $\a_g,$ on $\car$  we construct an idempotent mapping of an ultra-weakly dense subspace of $B(H \otimes \ell_2(G))$ onto    the crossed product von Neumann algebra $\cl(\car, \a, G),$ such that it is $\car-$bi-modular and  satisfies some nice relations with respect to positivity.  In the  case of an amenable group, our unbounded expectation turns into a usual  conditional expectation of norm 1.
 
\end{abstract}
\maketitle

\section{Introduction} We work in the setting of a von Neumann algebra  $\cas$ on a Hilbert space $K,$ such that $\cas$ is isomorphic to  a von Neumann algebraic reduced crossed product $\car \rtimes_\a G$ of an injective von Neumann algebra $\car$ by a countable discrete group $G.$  We give examples of unital completely positive mappings $\s: B(K) \to \cas$ which have lots of eigen vectors and in several ways have properties which  conditional expectations possesses. This has given us the courage to suggest the introduction of a new concept, we name {\em unbounded expectations.} 

The idea comes from the description of elements in a crossed product von Neumann algebra as generalized Fourier series. In our setting from above the crossed product algebra $\cas$ on $K$  is generated by an injective von Neumann algebra $\car$ on $K$ and a unitary  representation $g \to u_g$ of a countable discrete group $G$ on $K$ such that the mappings $\car \ni r \to \a_g(r) := u_g r u_g^*$  are *-automorphisms of $\car.$ It is well known that any element $s$ in $\cas$ has a generalized Fourier series $s \sim \sum_g u_gr_g$ with coefficients $r_g $ in $\car.$
The sum converges in some topology which we describe below, but just as in the ordinary case with the group $\bz$ and Fourier series for essentially bounded measurable functions on the interval $[-\pi, \pi],$ the meaning of convergence demands some background knowledge to be described. In the classical case with an essentially  bounded measurable  $2\pi$  periodic function $f,$ its   Fourier coefficients $f_n$ are determined by 
$$ f_n = \frac{1}{2\pi}  \int_{-\pi}^\pi e^{-in\theta} f(\theta)d\theta \text{ and } f \sim \sum_{ n \in \bz} f_n e^{in\theta}.$$ 
For the  $\cas,$ we study, there exists a faithful, unital, normal and  completely positive conditional expectation named Diag of $\cas $ onto $\car$ such that, with respect to a certain topology, the sum below converges, and we will use the symbol "$\sim$" to express this convergence.  
$$ \forall s \in \cas: \quad s \sim \sum_{g \in G} u_g\mathrm{Diag}(u_g^*s). $$ 
This is a generalization without restrictions of the classical result to our setting. In \cite{C2} we studied this generalization in the C*-algebraic setting, which in the case of periodic functions means the continuous functions, and we saw that the extension from the classical formula is more than just a notational coincidence. 
The same is true in the von Neumann case. In this article we will focus on the case where the von Neumann algebra $\car$ is injective. There are several reasons for this, and we will mention that the  the simplest  case with $\car = \bc I_K$ leaves us with the study of the von Neumann algebra $\cl(G)$ generated by the left regular representation of the group, and this is interesting enough for most  of us. On the other hand, many of the examples of von Neumann algebra factors  with certain properties are constructed as crossed products of an abelian von Neumann algebra by a discrete group, so they do fit into our setting too.

This article is based on an observation which tells, that in the setup described above, there exists a pair $(\c, \s)$ consisting of a positive definite and strictly positive  function $\c$ on the group and a  unital, completely positive and $\car-$bimodular  mapping  $\s: B(K) \to \cas,$  such that for any $g$ in  $G$ and any $r$ in $\car$  we have  $\s(u_gr) = \c(g)u_gr.$ Then all  the operators $u_gr$ are eigenvectors for $\s,$ with corresponding eigen values $\c(g),$ so  the observation may also be formulated as follows. There exists a pair $(\c, \s) $ such that $\c$ is a positive definite and strictly positive function on $G$ and $\s: B(K) \to \cas$ is a completely positive extension of the Hadamard multiplier $M^\c$ acting on $\cas.$ Let us then define  an ultra-weakly dense subspace  $\bd$ of $B(K)$ as the set of operators $x$ in $B(K)$ for which the sum $$ \sum_{g \in G}\c(g)^{-1} u_g \mathrm{Diag}\big(u_g^*\s(x)\big) $$ represents a bounded operator. We can then define an unbounded idempotent $\Pi$  of $\bd$ onto $\cas$ by 
$$ \forall d \in \bd : \quad \Pi(d) \sim \sum_{g \in G} u_g   \c(g)^{-1}  \mathrm{Diag}\big(u_g^*\s(x)\big) .$$ 

We will then call the pair $(\c, \s)$ for an unbounded expectation of $\bd$ onto $\cas.$ 
We show that unbounded expectations exist whenever the group $G$ is discrete and countable. For a discrete,  countable and  amenable group our construction of  an unbounded expectation $(\c, \s),$ gives a pair which  satisfies $\c(g) = 1 $ for all $g$ and then $\s $ is an ordinary completely positive conditional expectation from $B(K)$ onto $\cas. $  
In the classical case where the group  is $\bz$ and $\car$ the scalars,  our method recovers results on the convergence of Cesàro summation.  We make explicit computations for the free non abelian groups using methods from Haagerup's article \cite{H1}, and in this way we get some stronger results for these groups than our general results for discrete countable and finitely generated groups. 

The results leave the strong impression that amongst the unbounded expectations $(\c, \s),$ the one, with an associated positive definite  and strictly positive function $\c$ which is  point-wise the largest possible is the {\em best.} Recall that in the amenable case, $\c(g)=1$ for all $g$ is obtainable. 

The investigations behind this article are inspired by Haagerup's fundamental work \cite{H1}, and we have dedicated the work to his memory. 

We had the thought that the  unbounded expectations might have connections to Lance's and Kirchberg's works on the weak expectation property, or relations to Pisier's work on tensor products, or to Connes' embedding problem. We have  tried  to relate our findings to works by  Chatterji \cite{CI}, Kirchberg \cite{Ki}, Lance \cite{La}, Ozawa \cite{O1, O2} and Pisier \cite{P1},  
but we have not found any new insights into the fundamental problems dealt with by these authors.

 We recover a classical result by Cesàro on uniform convergence of the Cesàro sums for Fourier series of  continuous $2\pi$ periodic functions, and this is  related to  summation results by Toeplitz \cite{T} and Fejér \cite{Fe}.

The concept {\em rapid decay, } which Chatterji describes the properties of in \cite{CI} and which  plays an important role in Haagerups work \cite{H1}, seemed to have some impact on our investigation for some time, but in the end, we did not have to use any results depending on this concept. 
 
 Lance gives a characterization of nuclearity for a C*-algebra $\ca$  based on the existence of a net of finite rank, completely positive and contractive mappings converging norm point-wise to the identity mapping on $\ca.$ Our construction yields such a net for any group C*-algebra of an amenable group, but no new results for amenable groups are obtained. Lance's construction leads to the concept named {\em weak expectation, }  and this concept has been studied by especially Kirchberg \cite{Ki}, Ozawa \cite{O1, O2} and Pisier \cite{P1} because it relates to fundamental questions on tensor products of C*-algebras and also to Conne's embedding problem. It may be that our unbounded expectations could play a role in a future study, as weak unbounded expectations, but so far we have had no luck in pursuing such a track.

  As a final remark, we would like to mention that we have been worried by a statement, which we have heard M. Gromov has made: {\em "a theorem which is valid for all discrete groups is either trivial or false". } We think that the results presented here are true, and not for obvious reasons, but they  may not be extendible to discrete groups, which are not countable, so Gromov may still be right.

\section{ The algebra $\cl(\car, \a, G)$ and the Hadamard product}
\label{sec2}
We have for some time studied Schur block products on block matrices of bounded operators and Hadamard products in group C*-algebras \cite{C1, C2}. Our interest was inspired by an interest in the concept {\em spectral triple,} which is a fundamental concept in Connes' noncommutative geometry. We have lately, \cite{C2}, realized that the Schur block product may be seen as a special case of  the convolution or Hadamard product in a crossed product of a C*-algebra by a discrete group.
 When the elements in a crossed product of an algebra by a discrete group are described, via a unitary representation $g \to u_g,$ as formal sums $x \sim \sum_g u_gx_g,$ $y \sim \sum_g u_gy_g,$   the Hadamard product is given by the sum $x\star_h y: \sim \sum_g u_gx_gy_g.$ 
 The description in \cite{C2} of $M_n(B(H))$ as a crossed product shows that the diagonals  of the block matrices of $M_n(B(H))$ play a fundamental role in the study of properties of Schur block products. 
 In \cite{C2} the algebra $ M_n(B(H))$ is described as a crossed product of the diagonal algebra $\ell_\infty( \{0, \dots, n-1\}, B(H))$ with the natural action by cyclic permutations  of the cyclic cyclic group $C_n$ of order $n.$ Here the unitary named the forward shift $S$ in $M_n(\bc)$ is given by the matrix $s_{i,j} := \d_{i,j+1}$  and a unitary representation of $C_n$  on $H \otimes \bc^n$ is given by $i \to (I_H \otimes S^i).$ The unitaries $(I_H \otimes S^i)$ implement an action $\a_i$ of $C_n$ on the diagonal algebra such that $$ \a_i( \sum_j d_{(j,j)} \otimes e_{(j,j)} ) = \sum_j d_{(j-i,j-i)} \otimes e_{(j,j)}.$$  Then  any operator $A$ in $M_n(B(H)) $ may be  written as a member of the crossed product $C^*_r(\ell_\infty(\{0, 1, \dots, n-1\}, B(H)), C_n) $ by the formula 
 \begin{equation} \label{FouMn}
A = \sum_{i =0 }^{n-1}(I_H \otimes S^i)\text{Diag}\big((I_H \otimes S^{-i})A\big).
\end{equation}  
In the case when the cyclic group $C_n$ is replaced by the group $\bz$ and the algebra $M_n(B(H))$ is replaced by the group C*-algebra C$^*_r(\bz),$ which is isomorphic to the algebra C$(\bt)$ - the continuous functions on the unit circle - the formula (\ref{FouMn}) becomes the usual expression for the Fourier series of a continuous $2\pi-$periodic function $f( \theta)$ as \begin{equation} \label{FouZ} 
a_n : = \frac{1}{2\pi}\int_{-\pi}^{\pi} e^{-in\theta}f(\theta) d\theta, \quad f(\theta ) \sim \sum_{n \in \bz} a_n e^{in\theta}.
\end{equation} 
We would like to make a short historical remark,  which is based on Horn's survey \cite{Ho}.  Hadamard studied  in \cite{Ha} the product of two  Laurent series which is obtained as the series with coefficients equal to the product of the coefficients of the given series.  Schur studied in \cite{Sc} element-wise products of scalar matrices, and realized, amongst other results, that the Schur product of 2 positive semi definite matrices is positive semi definite, a result which plays an important role in both pure and applied mathematics. When reading Schur's article we noted that he mentions the ordinary matrix product too, and he uses the German word {\em faltung,} which means convolution for this product. In  a situation with a duality  present, it may be a matter of taste to decide which aspect is the primal and which is the dual. We have chosen to use the  term { \em Schur block product } when no group is present, and to say {\em Hadamard product }  when we are dealing with the {\em convolution product } inside a reduced crossed product of a C*-algebra by a discrete group.   

A word of warning. In the classical formula (\ref{FouZ}) the coefficients $a_n$ are written  to the left of the unitary $e^{in\theta}$ whereas in the operator form of the Fourier series presented in (\ref{FouMn}), the coefficients are written to the right of the unitary. This is so, and will remain the same in the rest of this article, because we find that this fits most easily with the traditional way of writing the polar decomposition, and it also  seems to fit best with the use of the forward shift in the description of $M_n(B(H))$ as a crossed product C*-algebra. 
In \cite{C2} we discussed several aspects of the Hadamard product in a reduced crossed product C*-algebra. 
This section is in many ways a continuation of this, but here we will work in the setting of an injective von Neumann algebra $\car$ on a Hilbert space $H$ with an action by *-automorphisms  $\a_g$ of a discrete group $G.$ 
The reduced crossed product von Neumann algebra $\cl(\car, \a,G)$ is constructed as in \cite{KR} Definition 13.1.1 with some notational differences, so $(\d_g)_{(g \in G)} $ denotes the canonical orthonormal basis for $\ell_2(G). $ The  unitary left-translations $\l_g$ satisfy $\l_g \d_h = \d_{gh}$ and the canonical matrix units in $B(\ell_2(G))$ are denoted $(e_{(g,h)})_{(g, h \in G)}.$ 
Then we define a unitary representation $ G \ni g \to L_g$ of  $G$ on  $H \otimes \ell_2(G)$ by $L_g := I_H \otimes \l_g, $ and we define a normal faithful representation $\Psi$ of $\car$ on the same Hilbert space by $\Psi(r) : =  \sum_{g \in G} \a_{g^{-1}}(r) \otimes e_{(g,g)}.$  
The von Neumann crossed product $\cl(\car, \a, G)$ is the von Neumann algebra which is obtained as the ultra-weak closure of the linear span of all the operators in the set $\{L_g\Psi(r) : g \in G, r \in \car\}.$ 
As in the definitions above we will freely use a block matrix representation of bounded operators on $H \otimes \ell_2(G)$ by writing them as infinite sums of elementary tensors such as $x = \sum_{g,h} x_{(g,h)} \otimes e_{(g,h)}.$ This structure is really that of the vector space of all functions $X(g,h)$  on $G \times G $ with values in $B(H)$ and we will occasionally use the expression $X = \sum_{g,h} X(g,h) \otimes e_{(g,h)}$ to describe such a function on $G \times G.$  
 We can then   define the operation of taking the main diagonal of a function with values in $B(H),$ or the diagonal of  a bounded operator   $x$ on $H \otimes \ell_2(G)$ by 
\begin{equation}
 \mathrm{Diag}(x): = \sum_{g \in G} x_{(g,g)}\otimes e_{(g,g)}.   
\end{equation}
Based on this we may generalize some of the results from \cite{C2} to the von Neumann algebra setting and define Fourier coefficients $x_g$ for an operator $x$ in the von Neumann crossed product algebra by \begin{equation}
\forall x \in \cl(\car, \a, G)\, \forall g \in G: \quad x_g:= \mathrm{Diag}(L_g^*x)
\end{equation}
We will write $x \sim\sum_g L_g x_g $ and it is a matter of computation to see that if $x$  is a finite sum of operators of the form $ x = \sum_{ g \in S}  L_g\Psi(r_g)$ then $x_g = \Psi(r_g).$ In \cite{C2} it was demonstrated that the Fourier Series for operators in a C*-algebraic reduced crossed product are convergent with respect to the topology given by the norm $\sqrt{\|\mathrm{Diag}(y^*y)\|}.$ In the von Neumann case it may not be like this,   but the series does converge with respect to the semi-norms $p_\xi$ on $B(H \otimes \ell_2(G))$ defined by 
\begin{equation}
\forall \xi \in H \otimes \ell_2(G)\, \forall x \in B(H\otimes \ell_2(G)): \quad p_\xi(x) := \|\big(\mathrm{Diag}(x^*x)\big)^{\frac{1}{2}}\xi\|.
\end{equation}   
It may not be obvious that $p_\xi$ is a semi-norm, but we may repeat the argument from the proof of Proposition 2.2 of \cite{C2} in a slightly different form and then remark, that since the mapping Diag is unital,  completely positive and normal, there exists a normal representation $\rho $ of $B(H \otimes \ell_2(G)) $ on a Hilbert space $K$ and an isometry $V: H \to K$ such that Diag$(x)  = V^*\rho(x)V, $ and then \begin{align*}
\forall x \in B(H &\otimes \ell_2G))\, \forall \xi \in H \otimes \ell_2(G) & \\ 
p_\xi(x) & = \|\big(\mathrm{Diag}(x^*x)\big)^{\frac{1}{2}}\xi\| \\
& = \|( V^*\rho(x^*x)V)^{\frac{1}{2}}\xi\| \\
& = \big(\langle V^*\rho(x^*x)V\xi, \xi\rangle \big)^{\frac{1}{2}} \\
&=\|\rho(x)V\xi\|,
\end{align*}  so it follows that $p_\xi$ is a semi-norm, which is continuous with respect to the ultra-strong topology on $B(H\otimes \ell_2(G)).$ It is easy to check that this family of semi-norms is separating, so the convergence of the series $\sum_g L_gx_g$ with respect to these semi-norms may be seen from the proof of Proposition 2.3  in \cite{C2}.
 It is now possible  to write 
down the generalization of the equations 
(\ref{FouMn}) and (\ref{FouZ})  such that the symbol $" \sim "$ means that the sum  converges in this topology.  
 
\begin{equation} \label{FouRG}
\forall x \in \cl(\car, \a, G): \quad x \sim \sum_{g \in G} L_gx_g.
\end{equation} 

When we look at elements in $B(H \otimes \ell_2(G))$ as block matrices indexed by elements in $G$ and with entries in $B(H)$ we may look at the Schur block  product defined by $$ A \square B=  (\sum_{g,h} a_{(g,h)} \otimes e_{(g,h)})\square (\sum_{g,h} b_{(g,h)} \otimes e_{(g,h)}):= \sum_{g,h} (a_{(g,h)}b_{(g,h)}) \otimes e_{(g,h)}
$$
When we restrict this product to the von Neumann algebra $\cl(\car, \a, G) $ we prefer to call it the Hadamard product and denote it $x\star_h y$ so we have 
\begin{equation}
\forall x, y \in \cl(\car, \a, G): \quad x\star_h y \in \cl(\car, \a, G)), \text{ and } x \star_h y \sim \sum_{g\in G} L_gx_gy_g.
\end{equation}
From here it is clear that the Hadamard product really wants us to focus on the operator Banach algebra -  not *-algebra -  $\ell_\infty(G, \car)$ equipped with the point-wise product. 
\begin{proposition} \label{fi}
The mapping $\f: \cl(\car, \a, G) \to \ell_\infty(G,\car) $ defined by 
$$\f(x)_g := \Psi^{-1}\big( \mathrm{Diag}(L_{g^{-1}}x)\big)$$ 
is a contraction and a  homomorphism of the algebra $( \cl(\car, \a, G), \star_h ) $ to the algebra $(\ell_\infty(G, \car), \cdot).$
\end{proposition}

We mentioned above Schur's fundamental result that the Schur product of 2 scalar and positive semi definite  matrices is positive semi definite. The same is true for the Hadamard product inside a crossed product of an {\em abelian }  C*-algebra by a discrete group. This follows from  more general results by   Jean Renault \cite{R1}, \cite{R2} on groupoids.  The works show  that the product of 2 functions of positive type on  a  groupoid is also a function of positive type. We want to be able to deal with the case where we study the crossed product of an abelian von Neumann algebra by a discrete group and also with the case where a   Schur multiplier with scalar entries acts on block matrices over a C*-algebra. Both results are known, but we have a simple lemma, which is a bit more general, and  covers both situations.  
\begin{lemma} \label{PosProd} 
Let $\cc$ be a C*-algebra with a pair of commuting sub C*-algebras $\ca, \cb$ and $n$ a natural number. Let $A =(a_{(i,j)})$ be a positive operator in $M_n(\ca)$  and $B= (b_{(i,j)})$ a positive operator in $M_n(\cb).$ The Schur block product $C:= A \square B := (a_{(i,j)}b_{(i,j)}) $ is positive in $M_n(\cc).$ 
\end{lemma}
\begin{proof}
The algebras are C*-algebras, so positive operators have positive square roots, and let $X := \sqrt{A} = (x_{(i,j)}), Y:= \sqrt{B} =(y_{(i,j)}). $  We will  write $X = \sum_{i,j} x_{(i,j)} \otimes e_{(i,j)}$ and $Y = \sum_{i,j} y_{(i,j)} \otimes e_{(i,j)},$ and ahead of the computations to come we define for each pair $s,t$ of integers in  $\{1, \dots, n\}$ the operator $z_{(s,t)} := \sum_{j=1}^n y_{(t,j)}x_{(s,j)} \otimes e_{(1,j)}.$  We can then start computing  
\begin{align*}
C &= \sum_{i,j = 1}^n a_{(i,j)}b_{(i,j)} \otimes e_{(i,j)} \\
& = \sum_{i,j = 1}^n\bigg( \sum_{s=1}^n\sum_{t=1}^n x_{(s,i)}^*x_{(s,j)}y_{(t,i)}^*y_{(t,j)}\bigg) \otimes e_{(i,j)} \\
& \quad \text{ since the algebras } \ca \text{ and } \cb \text{ commute } \\ 
&=\sum_{i,j = 1}^n\bigg( \sum_{s=1}^n\sum_{t=1}^n x_{(s,i)}^*y_{(t,i)}^*y_{(t,j)}x_{(s,j)}\bigg) \otimes e_{(i,j)} \\
&= \sum_{s=1}^n\sum_{t=1}^n\bigg(\sum_{i = 1}^n x_{(s,i)}^*y_{(t,i)}^*\otimes e_{(i,1)}\bigg)\bigg(\sum_{j =1}^n y_{(t,j)}x_{(s,j)} \otimes e_{(1,j)}\bigg)  \\
&= \sum_{s=1}^n\sum_{t=1}^n z_{(s,t)}^* z_{(s,t)} \, \geq 0.
\end{align*}
\end{proof}

\section{ Completely positive, unital and $\Psi(\car)$  bimodular mappings from $B(H \otimes \ell_2(G))$ to $\cl(\car, \a, G)$  }

We start with a construction, which to any unit vector $\xi$ in $\ell_2(G)$ associates a unital  completely positive and $\Psi(\car)$ bimodular  mapping $\s_\xi$  of $B(H \otimes \ell_2(G))$ into  $\cl(R, \a, G).$

We will now recall the setup of Section \ref{sec2}, so $\car$ is an injective von Neumann algebra acting on a Hilbert space $H$ and the discrete group $G$ acts by *-automorphisms $\a_g$ on $\car. $
A positive definite function $\c(g)$ on $G$ with $\c(e) = 1$  defines a Hadamard multiplier $M^\c$ on $\cl(\car,\a, G) $ by $$M^\c\star_h\sum_g L_g\Psi(r_g) ):= \sum_g \c(g)L_g \Psi(r_g),$$ and we know,  Lemma \ref{PosProd}, that the multiplication mapping $M^\c$ is unital and completely positive.  The operators of the form $L_g\Psi(r)$ are eigen vectors with the eigen value $\c(g).$   
We will fix  a completely positive conditional expectation $\pi$ of $B(H)$ onto $\car.$  Then, to any unit vector $\xi$ in $\ell_2(G)$  we can construct a completely positive unital and $\Psi(\car)$ bimodular   mapping $\s_\xi$ of $B(H \otimes \ell_2(G))$ into $\cl(\car, \a, G),$ such that for the positive definite function $\c(g) := \langle \l_g \xi, \xi\rangle$ the mapping $\s_\xi$ is an extension of the Hadamard multiplier $M^\c.$

\begin{theorem} \label{sigma}
Let $\xi = (k_g)_{( g \in G)}$ be a unit vector in $\ell_2(G),$ then the  mapping $\s_\xi$ of $B(H\otimes \ell_2(G))$ to $\cl(\car, \a, G)$    defined by the equation
\begin{equation} \label{sigmaS}
\s_\xi(x) := \sum_{g,h \in G}L_{gh^{-1}}\Psi\big(\a_h(\pi(x_{(g,h)})\big)\overline{k_g}k_h 
\end{equation}
is unital, completely positive, $\Psi(\car)$ bimodular  and any operator of the form $L_g\Psi(r)$ is an eigen vector with the eigen value $\langle \l_g \xi, \xi\rangle.$ 
\end{theorem}

\begin{proof}
Let $x = \sum_{g,h} x_{(g,h)} \otimes e_{(g,h)}$ be a bounded operator on $H \otimes \ell_2(G)$ then we may consider $x$ as an element in $\ell_\infty(G \times G, B(H)),$ and in this setting we can see that the definition of $\s_\xi$ is meaningful as a contraction operator on $\ell_\infty(G \times G, B(H)). $ To see this, we  will make a change of variables in the definition of $\s_\xi$ such that $h=h$ and $t = gh^{-1},$ then 
\begin{equation}  \label{sinfty}
\s_\xi(x) := \sum_{t,h \in G}L_{t}\Psi\big(\a_h(\pi(x_{(th,h)})\big)
\overline{k_{th}}k_h. 
\end{equation} We know that the product of the two $\ell_2$ families $(k_h)_{(h \in G)}$ and  $(\overline{k_{th}})_{(h \in G)}$ is in $\ell_1(G)$ so $\|\s_\xi(x)\|_\infty \leq \|x\|_\infty.$  It is worth to remember in the following computations that whenever we make a change of variables or change the order of summation, then at each entry indexed by a pair $(g,h)$ we are dealing with a sum of a family, which is element-wise dominated in norm by a positive $\ell_1(G)$ family, so the sums  stays well defined at each entry in $\ell_\infty(G \times G, B(H)).$ When we insert the formulas   $\Psi(r) = \sum_v \a_{v^{-1}}(r) \otimes e_{(v,v)} $ and $L_t = \sum_u I \otimes e_{(tu,u)}$ in the definition of $\s_\xi$  we get

\begin{align} \label{cp}
\s_\xi(x) & = \sum_{g,h\in G} \sum_{u, v \in G}  
\big(I \otimes e_{(gh^{-1}u,u)}\big)\big(\a_{v^{-1}h}(\pi(x_{(g,h)} )) \otimes e_{(v,v)}\big)\overline{k_g}k_h  \\
& = \sum_{g,h\in G} \sum_{v\in G}  
\a_{v^{-1}h}(\pi(x_{(g,h)} )) \otimes e_{(gh^{-1}v,v)}\overline{k_g}k_h \notag \\ & \text{change variables } g =g ,\, h=h, \, u = v^{-1}h \notag \\
& =\sum_{u \in G} \sum_{g,h \in G}  
\a_{u}(\pi(x_{(g,h)} )) \otimes e_{(gu^{-1},hu^{-1})}\overline{k_g}k_h \notag \end{align} 
We will now fix a group element $u$ and let $\rho_u$ denote the unitary right translation on $\ell_2(G)$ given by $\rho_u \d_g := \d_{gu^{-1}}.$ In matrix form we get $\rho_u = \sum_t e_{(tu^{-1}, t)}.$ We will define the diagonal operator $M_\xi$ on $\ell_2(G)$ by $M_\xi := \sum_g k_h e_{(h,h)},$ then we can  define a completely positive unital map $\t_u $ of $B(H \otimes \ell_2(G))$ into $B(H \otimes \ell_2(G))$ by 
\begin{equation} \notag
\t_u(x) = (I_H \otimes \rho_u) (I_H \otimes M_\xi^*) \bigg(\big((\a_u \circ \pi)\otimes {\mathrm id}_{B(\ell_2(G))}\big)(x)\bigg)(I_H \otimes M_\xi) (I_H \otimes \rho_u^*)
\end{equation}
and in matrix form we get 
\begin{align*}
\t_u(x) & = \sum_{s, t \in G} \sum_{g,h \in G}\sum_{i,j \in G} (I_H \otimes e_{(su^{-1},s)})(I_H \otimes \overline{k_g}e_{(g,g)} )\\ & \quad \quad  \bigg(a_u(\pi(x_{(i,j)}))\otimes e_{(i,j)} \bigg) (I_H \otimes k_he_{(h,h)} )(I_H \otimes e_{(tu, t)}) \\ & \text{ since } s = g = i \text{ and } j =h=  tu \\
&= \sum_{g,h \in G}  \overline{k_g}k_h\a_u(\pi(x_{(g,h)} )) \otimes e_{(gu^{-1}, hu^{-1})}.
\end{align*}
From (\ref{cp})  we may then  conclude that $\s_\xi$ is completely positive as a sum of completely positive mappings. One may be worried about the infinite sum over $u$ in $G$ but a simple computation shows that for the unit we get $$ \s_\xi(I_{H \otimes \ell_2(G)})  = \sum_{g \in G} \sum_{u \in G} |k_g|^2 I_H \otimes e_{(gu^{-1}, gu^{-1})}= I_{H \otimes \ell_2(G)}$$ so for a positive bounded operator $x$ the finite  sums over $u$ form a bounded increasing net, and such a net is known to be ultra-strongly convergent. In particular the sum on each entry $(g,h) $ is convergent towards the entry of the ultra-strongly convergent sum, so $\rho_\xi$ is unital and completely positive.

 We show   by a direct computation that any operator of the form $L_t\Psi(r)$ is an eigen vector for $\s_\xi$ with the eigen value $\langle \l_t \xi, \xi\rangle,$  and remark first that  $L_t\Psi(r)$ in matrix form is given as \begin{equation} \label{LtPsiMatrix} L_t\Psi(r)  = \sum_{u,h \in G}(I \otimes e_{(tu,u)})(\a_{h^{-1}}(r)\otimes e_{(h,h)}) = \sum_{h\in G} \a_{h^{-1}}(r) \otimes e_{(th, h)}, \end{equation} so \begin{equation} \label{(g,h)}
  L_t\Psi(r)_{(g,h)} = \begin{cases}
  0 \, \, \, \text{ if } g \neq th \\
  \a_{h^{-1}}(r) \text{ if } g = th.
  \end{cases}  
 \end{equation}
The matrix formulation of $\s_\xi$ from equation (\ref{cp}) may then be applied and we get
\begin{align} \label{eigenvalue}
\s_\xi(L_t\Psi(r)) = & \sum_{u \in G}\sum_{ h \in G}  \overline{k_{th}}k_h \a_{uh^{-1}}(r) \otimes e_{(thu^{-1},hu^{-1})}, \text{ by } (\ref{LtPsiMatrix}) \notag \\   =& \sum_{h \in G}  \overline{k_{th}}k_h L_t\Psi(r) \notag \\ =&
\langle \l_t \xi, \xi\rangle L_t(\Psi(r).
\end{align} 

To see that $\s_\xi$ is right $\Psi(\car)$ modular we remark first that for an $x$ in $B(H \otimes \ell_2(G))$ and an $r$ in $\car$ we have 
\begin{align*} x\Psi(r) &= \sum_{g,h , v \in G }( x_{(g,h)} \otimes e_{(g,h)} )(\a_{v^{-1}}(r) \otimes e_{(v,v)} \\& = \sum_{g,h \in G} x_{(g,h)} \a_{h^{-1}}(r) \otimes e_{(g,h)}.
\end{align*}
It is known that  $\pi$ is  $\car$ bimodular, so $$\pi\big(x_{(g,h)}\a_{v^{-1}}(r)\big) = \pi\big(x_{(g,h)}\big)\a_{v^{-1}}(r),$$ and then the eigenvalue properties of $\s_\xi$ plus  a repetition of the computations from above yield the last steps in establishing the right $\Psi(\car)$ modularity. The left $\Psi(\car)$ modularity follows in a similar, but slightly more complicated way, and the proposition follows.
\end{proof}

\begin{definition} \label{chiFin} Let $G$ be a discrete group and $S$ a finite non empty  subset of $G$. The  function $\c_S$ on $G$ is defined by \begin{equation}\label{fS}
\c_S(g):= \frac{| S \cap gS|}{|S|}.
\end{equation}
\end{definition}
The following proposition follows from  Hulanicki's work \cite{Hu} on invariant means.
  
\begin{proposition} \label{chiS}
Let $G$ be a discrete group and $S$ a finite non empty subset of $G,$ The function $\c_S$ is  positive definite on $G$ with $\c_S(e) = 1$ and support equal to $\{gh^{-1}\, : \, g, h \in S\, \},$ and for $\xi_S$ being the normalized characteristic function for $S$ in $\ell_2(G): $ $$ \c_S(g) = \langle \l_g \xi_S, \xi_S \rangle.$$  
\end{proposition}
\begin{proof}
The proof is part of the proof of  Proposition 6.1 of  \cite{Hu}.   
\end{proof} 

For a discrete countable group $ G$ there exists a F{\o}lner sequence of subsets $(S_n)$ of $G$, see Section \ref{aG}, and Hulanicki uses the positive definite functions $\c_{S_n} $ to obtain the positive definte function $\c(g) = 1 $ for all $g$ as a limit point. Based on Theorem \ref{sigma} we will follow Hulanicki's idea and look at a sequence of subsets $(S_n)$ of finite subsets of a discrete countable group $G,$ which acts on an injective von Neumann algebra $\car$ and then study the set of possible limit  points for the unital completely positive mappings $\s_{\xi_n},$ 
where $\xi_n$ is the normalized characteristic function for the set $S_n.$
\begin{theorem} \label{vwep}
Let $G$ be a  discrete group which acts on an injective von Neumann algebra $\car$ on a Hilbert space $H$ and
let $\pi$ denote a completely positive projection of $B(H)$ onto $\car.$ To any  sequence  $(B_n)_{(n \in \bn)} $ of finite subsets of $G$  there exists a subnet $J \ni \iota \to n(\iota) \in \bn$ such that the net of positive definite functions  $(\c_{B_{n(\iota)}}) $ converges point-wise to a positive definite, and positive  function $\c$ on $G.$ The net of unital completely positive mappings $\s_{\xi_{n(\iota)}}$ converges point-wise ultra-weakly to a unital completely positive and $\Psi(\car)$ bimodular mapping $\s$ of $B(H \otimes \ell_2(G))$ to $\cl(\car ,  A, G), $ such that
\begin {itemize}
\item[(i)]
 \label{eigenval} 
 $\forall g \in G \forall r \in \car : \, \s(L_g\Psi(r)) = \c(g)L_g\Psi(r)$ 
\item[(ii)] $\forall x \in \cl(\car, \a,G):  \, \s(x) = M^{\c} *_h x.$
\item[(iii)]  
 For each $x$ in the norm closure of $\mathrm{span}(\{L_g\Psi(r): g \in G, \, r \in \car\}),$ the net $(\s_{n(\j)}(x))_{( \j \in J)}$ converges in norm  to $\s(x).$ 
\end{itemize} \end{theorem} 

\begin{proof}
Since the unit ball of $\cl(\car, \a, G)$ is ultra-weakly compact and all  the mappings $\s_{\xi_n} $ are  unital completely positive and $\Psi(\car) $ bimodular, we may find a  point-wise ultra-weakly convergent subnet $(\s_{\xi_{n(\j)}})_{\j \in J}$ of the sequence which converges to a unital completely positive and $\Psi(\car)$ bimodular  mapping $\s$ of $B(H \otimes \ell_2(G)) $ into $\cl(\car, \a, G).$ 
 Each $L_g$ is an eigen vector for each $\s_{\xi_n},$ so we find that $\s_{\xi_n(\j)}(L_g) $ converges in norm for $\j \in J$ to a multiple, say  $\c(g)$ of $L_g,$ and  $\s(L_g) = \c(g) L_g.$  The net of positive definite and positive  functions $(\c_{\xi_{n(\j)}})$ converges point-wise to $\c,$ so $\c$ is positive definite, positive  and satisfies $\c(e) =1.$ We then also have $\s(L_g\Psi(r)) = \c(g) L_g \Psi(r) $ for $g$ in $G$ and $r$ in $\car $ so claim (i) and (ii) follows.

 To see that  the subnet $\s_{\xi_{n(\j)}}(x)$ converges in norm to $\s(x)$ when $ x$ is in the norm closure of the linear span of operators of the form $L_g\Psi(r)$ we remark that to such an operator $x$ and a positive $\eps$  there exists an operator $x_0 = \sum_{l=1}^k L_{g_l} \Psi(r_l)$   such that $\|x- x_0\| < \eps/3.$ We know that for each $l$ with $ 1 \leq l \leq k$ will  $\c_{n(\j)}(g_l)$ converge to $ \c(g_l),$   so we may, by (i),  find a $\j_0$ in $J$ such that  $$ \forall  \j > \j_0: \quad \|  \s_{\xi_{n(\j)}} ( x_0 ) - \s(x_0)\| < \eps/3$$ and  
\begin{align*} \forall \j \geq \j_0:\quad &\|\s(x) - \s_{\xi_{n(\j)}}(x) \| \leq \|\s(x- x_0) \| \\& + \|\s(x_0) - \s_{\xi_{n(\j)}}(x_0)\| + \| \s_{\xi_{n(\j)}}(x_0 -x)\| \leq \eps, \end{align*}  
 and the proposition follows. 
\end{proof}

We find that it is worth to mention, that the property (iii) is inspired   by Haaagerup's arguments in the beginning of the proof of his Theorem 1.8 of \cite{Ha}. We will see that in the case of the group $\bz$ and $\car = \bc$ the content of item (iii) tells that the Cesàro means of a  continuous  $2\pi$ periodic function converges in norm to the function. 

It  might happen that the positive definite function $\c(g)$ of Proposition \ref{vwep} turns out to be  be trivial in the sense that $\c(g) = \d(g,e),$ or it might be that $\c(g)$ vanishes on many group elements as for instance $\c_B$ does, when $B$ is a finite set. For theoretical reasons we are interested in finding unital completely positive $\Psi(\car)$ bimodular  mappings $\s$ with associated positive definite and strictly positive  eigen value functions $\c(g),$ which are as big as possible. One reason is, that for amenable, countable discrete groups  it is possible to find $\c$ with $\c(g) = 1$ for all $g.$ Another reason is that we have got the impression during our work with this article, that the dual algebra aspect, which is seen in Proposition \ref{fi}, is important in this setting. The order in $\ell_\infty(G, \bc)$ is that  of point-wise order , so we think it may be reasonable to study the cone  of non negative positive definite functions on $G$ equipped with the point-wise order instead of the cone of positive definite functions equipped with the order given by positive definiteness.  We will end this section with the vague statement that the bigger a $\c(g)$ you can obtain for a group, the {\em more amenable } the group will be. In the next section we will make computations for several groups.

\section{Examples of completely positive mappings of $B(H \otimes \ell_2(G))$ into $\cl(\car, \a, G)$}  \label{Exa}
In order to see that the theorems   \ref{sigma} and  \ref{vwep} may offer non-trivial pairs $(\c, \s) $ of a positive definite and non-negative  function $\c$ on $G$ and a unital completely positive mapping $\s$ such that $\s(L_g\Psi(r)) = \c(g)L_g\Psi(r)$   for all pairs $g$ and $ r,$  we have computed some examples. It is worth to remember that the injective von Neumann algebra $\car $ and the action $\a_g$ by $G$ on $\car$ only plays the very  indirect role in the construction of $\c,$ namely that these ingredients may have a strong influence on which point-wise convergent  subnet $\c_{B_{n(\j)}} $ the situation will choose, while $\c_{B_{n(\j)}} $  only depends on $B_{n(\j)}.$
\subsection{ A positive $\xi$}
It is rather obvious that for a unit vector $\xi = (k_g)_{(g \in G)}$ with $k_g \geq 0$ for all $g$ we get that the associated positive definite function $\c_\xi(g) := \langle \l_g \xi, \xi \rangle$ is non-negative.

 \subsection{{\bf The integers  $G = \bz$}} 
In this case we will consider the sequence $(B_n)_{n \in \bn_0} $ of subsets of $\bz$ given by $B_n = \{0, 1, \dots, n\},$ and we find that for any $k$ in $\bz$ 
$$|\{0, \dots, n\}\cap\{k , \dots, k+n\}| = \begin{cases} 0 \text{ if } |k| > n \\
n+1 - |k| \text{ if } |k| \leq n. \end{cases}$$
Then from Definition \ref{chiFin} we get
\begin{equation}
\forall k \in \bz: \quad \c_{B_n}(k) = \begin{cases} \quad 0 \quad \quad \text{ if } |k| > n \\
\frac{n +1 - |k|}{n +1}  \,\text{ if } |k| \leq n.\end{cases}
\end{equation} 
When $x$ is in $C^*_r(\bz) ,$ which is isomorphic to $C(\bt),$ we write \newline
$ x \sim \sum_{k \in \bz} c_k z^k $ and we get 
\begin{equation} \s_{\xi_n} (x)  =  \sum_{k = - n }^{n} \frac{n+1 - |k|}{n+1} c_k z^k,
\end{equation} 
and this is the $n'$th Cesàro mean of the continuous function $x(z).$ In this case the classical theory on Fourier series tells us that for a continuous function  $f$ on $\bt$  the Cesàro means converge uniformly to  $f.$ 
It is clear that $\c_{B_n}(k) \to 1 $ for all $k \in \bz$ and  $n \to \infty,$
so the limit mapping $\s$ is a completely positive idempotent of $B(\ell_2(\bz))$ onto $\cl(\bz).$

\subsection{{\bf A discrete, countable  and amenable group $G$}} \label{aG}
A   discrete, countable and  amenable group contains a F{\o}lner sequence, which means a sequence $(B_n)_{( n \in \bn)}$ such that for each $t$ in $G,$ the number of elements in the symmetric difference $tB_n \Delta B_n$ becomes small when compared to the number of elements in $B_n, $  i.e. for any group element $t,$

$$ \underset{n \to \infty}{\lim}\frac{|tB_n \Delta B_n|}{|B_n|} =  0.$$

In the construction of $\s_{\xi_n}$
any $\l_t$ is an eigen vector and the eigenvalue, $\c_n(t)$ is computed as follows 
\begin{align*}
|B_n| & = |B_n \cap tB_n| + \frac{1}{2}|tB_n \Delta B_n| \\
 \c_n(t) &= \frac{ | B_n \cap tB_n|}{|B_n|} \\
& = 1 - \frac{1}{2} \frac{|B_n \Delta tB_n|}{|B_n| } \\
& \to 1 \text{ for } n \to \infty.
\end{align*} 
Hence $(\s_{\xi_n})$ is a sequence of finite rank, unital completely positive mappings of $B(\ell_2(G))$ into $\cl_r(G)$ which converges point ultra-weakly to a conditional expectati½on onto the group von Neumann algebra. $\cl(G).$

\subsection{{The noncommutative free groups \bf $\baf(k)$}} \label{Fk} In this case we will use the canonical length function on the group given as the word length of a group element written in the alphabet consisting of the $2k$ letters  $\{a_k^{-1}, \dots , a_1^{-1}, a_1, \dots , a_k\}.$ The sequence of finite sets $B_n$ begins with $B_0 = \{e\} ,$ the ball of radius $0,$ but we will also denote it  $S_0,$ which means the sphere of $B_0.$ It  continues such that $B_n$ consists of all group elements  of length at most $n,$ and $S_n$ denotes all the elements of length n. 
First we compute the sizes of $B_n$ and $S_n$ and remrak that $|B_0| = 1,$  $| S_1 | = 2k$ and inductively  \begin{equation}
\forall n \geq 1 : \quad |S_n| = 2k(2k-1)^{(n-1)}. 
\end{equation}
  Hence for $n \geq 1 $ we have
\begin{equation} \label{Bn}
|B_n| = 1 + 2k\sum_{j=0}^{n-1}(2k-1)^j = 1 + 2k\frac{(2k-1)^n -1}{2k -2} = \frac{k(2k-1)^n -1}{k-1}.
\end{equation}
Given a group element $t,$ then in order to compute the eigen value  for $\l_t$ with respect to  $\s_{\xi_n}$ we have to count the number of elements in the set $T_n(t) $ which is defined by $$ T_n(t) := \{h \in B_n :  th \in B_n\}. $$  For us it seems to be the easiest to compute this value based on the knowledge of the length $\ell$ of $t.$
 Since $n$ has to increase in order to get to the limit $\s$ we may as well assume that $n$ is at least twice as big as $\ell. $ We may then split $T_n(t) $ in $\ell + 1$ disjoint subsets, and we will compute the size of each of these and then add the sizes in the end.
  We write $$ T_n(t) = \big(T_n(t)\cap B_{(n-\ell)}\big) \bigcup_{i =0}^{\ell -1} \big(T_n(t)\cap S_{(n-i)}\big).$$
The first set is easy to count since $B_{(n-\ell)} \subseteq T_n(t),$
so \begin{equation} \label{n-ell}
|T_n(t) \cap B_{(n-\ell)} | = \frac{k(2k-1)^{(n-\ell)} -1}{k-1}.
\end{equation}
Let us choose an index $i$, $0 \leq i \leq \ell -1,$ and a group element  $h$ in $S_{(n-i)},$ then if no cancellations occur  the length of $th$ will exceed $n$ so $h$ is not in $T_n(t)$ we must then investigate the possible cancellations which will imply that $h$ is in $T_n(t).$ 
To this end we write $th$ in the alphabet as $t = t_{\ell} \dots t_2t_1 h_1h_2 \dots h_{(n-i)}$ and we see that the reduced word will be in $B_n$ if and only if there are at least $p$ cancellations,
 where $p$  is determined as the least integer such that $\ell + n - i -2p \leq n,$  which means that $p$ is the least integer  such that $p \geq   \frac{\ell - i}{2}.$
  According to Knuth's notation $p = \lceil  \frac{\ell  -i}{2} \rceil.$ 
  The group elements $h $ in $T_n(t) \cap S_{(n-i)} $ are then characterized by the property that the first $p$ letters are fixed and the following $n -i - p$ letters are chosen such that no cancellations occur and the length of the word then is $(n-i). $
   This means that the number of group elements in $T_n(t) \cap S_{(n-i)}$ is $(2k-1)^{( n- i - p) } .$
   The exponent can be reformulated as follows $$n-i - p = (n- \ell)  + ( \ell -i - \lceil \frac{\ell - i}{2}\rceil = n - \ell  + \lfloor   \frac{\ell - i}{2}\rfloor,$$ where $\lfloor x \rfloor$ is the largest integer smaller than or equal to  $x,$ and then   \begin{equation} \label{n-i}
   |T_n(t) \cap S_{(n-i)} | = (2k-1)^{(n-\ell + \lfloor \frac{\ell - i}{2})\rfloor)}.
\end{equation} 
The    use of the floor function $\lfloor x \rfloor$ in (\ref{n-i}) implies that computation of the value $\c(t)$  will be different depending on whether the length $ \ell = \ell(t)$ is even or odd. Let us first assume that $\ell$ is even, so  $\ell = 2m$ for an $m $ in $\bn_0.$ 
Then for typographical reasons define $q : = (2k-1)$ and we find 
  \begin{align}  &\, \, \, \, \,\,  \sum_{i = 0}^{2m - 1} (2k-1)^{(n-2m + \lfloor \frac{2m - i}
  {2}\rfloor)} =
  \sum_{i = 0}^{2m - 1} q^{(n-2m + \lfloor \frac{2m - i}
  {2}\rfloor)}\\
  \notag & = q^{(n -m)} + 2q^{(n -m+1)} + 
 \dots + 2q^{(n -2m+1)} +q^{(n -2m)} \\
 \notag & = \frac{1}{q -1}\bigg( q^{(n-m+1)} + 
q^{(n-m)} - q^{(n-2m+1)} -q^{(n-2m)}\bigg).
 \end{align}
When expressed in $q$ the equation (\ref{n-ell}) tells that \begin{align} \label{Bnq}
|B_{n -2m} | &= (q^{(n-2m+1)} + q^{(n-2m)} -2)/(q-1) \\
|B_{n} | &= (q^{(n+1)} + q^{(n)} -2)/(q-1)\end{align}  so by (\ref{Bn}) \begin{align}  \notag
\text{ for } \ell(t) = 2m: |T_n(t)| &= \frac{q^{(n-m+1)} + q^{(n-m)} -2} {q-1}  & \\ \label{2m}
\underset{n \to \infty}{\lim} \frac{|T_n(t)|}{|B_n| } &= \underset{n \to \infty}{\lim} \frac{q^{(n-m+1)} + q^{(n-m)} -2}{q^{(n+1)} +q^n -2}\\  \notag
&= q^{-m}\\ 
\text{ for } \ell(t) = 2m: \c(t) & = \frac{1}{(2k-1)^m}.
\end{align}
For $\ell(t) = 2m+1$ with $m \in\bn_0$ a similar computation gives 
\begin{align}  \notag
\text{ for } \ell(t) = 2m+1: |T_n(t)| &= \frac{2q^{(n-m)} -2}{q-1}\\ \label{2m+1}
\underset{n \to \infty}{\lim} \frac{|T_n(t)|}{|B_n| } &= \underset{n \to \infty}{\lim} \frac{2q^{(n-m)} -2}{q^{(n+1)} +q^n -2}\\  \notag
&= \frac{2}{q+1}q^{-m}\\ 
\text{ for } \ell(t) = 2m+1: \c(t) & = \frac{1}{k(2k-1)^m}.
\end{align}
One may wonder if there is an easy explanation of the relatively nice formulas (\ref{2m}) and (\ref{2m+1}) ?
It is worth to try to compare the function $\c$ to Haagerup's multiplier \cite{Ha} $\f_\eps( t) := e^{-\eps\ell(t)}$ and for $\eps = \frac{1}{2}\log(2k-1)$ it fits if the length is even, and the ratio between the two functions is about $\sqrt{k/2}$ when the length is odd, so the 2 multipliers have very similar decay properties.

\subsection{A discrete, countable and   finitely generated group} \label{FinGen}  
 We consider a general countable discrete group $G$ which is generated by a finite set $S$ such that $e \notin S$ and $S^{-1} = S.$ The generators induce a natural length function $\ell(g)$  on $G$ defined by word length in the number of generators. As for the free non abelian groups above, we will let the sequence of subsets $B_n$ be defined  by $B_0 = \{e\},$ and for $n \in \bn, $ $B_n := \{g \in G:   \ell(g) \leq n\}.$ Then in order to study a point ultraweak limit $\s$   of the mappings $\s_{\xi_n}$ we can give a lower bound for the positive definite function $\c(g)$ which gives the eigen vaules of the eigenvectors $\l_g$ for $\s.$ For an amenable group the example above shows that it is possible to obtain  $\c(g) = 1.$ The worst case would be if for some group $G$ the only positive definite function $\c$  our construction can provide is the function given by $\c(g) = 0$ for $g \neq e$ and $\c(e) = 1.$ If the group $ G$ has a finite symmetric set with $k$ generators, then we can find a completely positive mapping $\s$ such that the associated positive definite function satisfies \begin{equation} \label{chi}
\forall g \in G : \c(g) \geq \frac{k-2}{(k-1)^{\ell(g) +1} -1}.
\end{equation}
When compared to the results above for the free non abelian groups we see that the decay of $\c(g),$ there, for approximately half the numbers of generators,  is of the order $(k-1)^{\ell(g)/2},$ but it is somehow surprising that no information on the group except being countable, discrete and having a finite, symmetric generating set of size $k$  gives existence of a pair $(\c, \s)$  with a decay rate
for $\c(g)$ which is only the square of the one from  a non abelian free group with approximately half the number of generators.  The computation of the lower bound is quite simple. 
Let $g$ in $G$ have length $\ell,$ then for $n \geq \ell$ we want to compute a lower bound for $\c_n(g)  =  |B_n \cap g B_n|/ |B_n| .$ First we remark that \begin{equation} \label{nom} B_n \cap gB_n \supset B_{(n - \ell)}.
\end{equation} 
Secondly we define the spheres $S_m = \{g \in G : \ell(g) = m\}$ and remark, as above for the free groups, that $B_n = B_{(n - \ell)} \cup S_{(n - \ell +1)} \cup\dots \cup S_n.$ It is easy to see that $|S_{(j+1)} | \leq (k-1)|S_j| \leq (k-1) |B_j|.$  \begin{equation} \label{denom}
|B_n|  \leq |B_{(n - \ell)}|\big( 1 + (k-1) + \dots + (k-1)^{\ell}\big) = |B_{(n - \ell) }|\frac{(k-1)^{(\ell+1)} -1}{k-2}. 
\end{equation}
Hence for any $g$ in $G$ with $\ell(g) \leq n$ 
\begin{equation}
 \frac{|B_n \cap g B_n|}{|B_n|}  \geq \frac{|B_{( n - \ell(g))}|(k-2)} {|B_{(n- \ell(g))}|( (k-1)^{(\ell+1)} -1)} = \frac{(k-2)} { (k-1)^{(\ell+1)} -1},
\end{equation}
and it follows that the limiting positive definite function$\c(g)$ satisfies $\c(g) \geq \frac{(k-2)} { (k-1)^{(\ell(g)+1)} -1}.$

\subsection{A discrete countable  group} \label{Countab} 

Now,  we just assume that $G$ is countable, and  then  there are lots of unit vectors $\xi$ in $\ell_2(G)$ which are strictly positive so the corresponding positive definite function $(\xi\star \xi^*)(g^{-1}) = \langle \l_g \xi, \xi\rangle $ is also strictly positive with value 1 at the unit $e.$ By taking convex combinations and point-wise limits of such positive definite and strictly positive functions we can obtain a lot of positive definite functions $\c$ which have the property that there exists a unital completely positive and $\Psi(\car)$ bimodular extension $\s$  of the Hadamard multiplier $M^\c$ acting on $\cl(\car, \a, G)$ to a mapping from $B(H \otimes \ell_2(G))$ to $\cl(\car, \a, G)$ The question is, which positive definite functions on the group have this property ?

\subsection{ Remark   }
Uffe Haagerup shows in \cite{H1}
that for the free non abelian group of $n$ generators with length function $\ell$ and any $\eps > 0,$ the function $\c(g) := e^{-\eps \ell(g)}$ is a positive definite function on the group, and for $\eps \to 0$ it  is an approximate unit in $C^*_r(G) $ with respect to  Hadamard  multiplication. 
It would be interesting to see, if for some other  countable groups there exists a positive definite  function $\c$ which  has the property that for any $\eps$ in $]0, 1]$  the function  $\c^\eps(g) := (\c(g))^\eps$ is positive definite and for some $\eps_0$  the associated Hadamard multiplier extends to a completely positive mapping on $B(H \otimes \ell_2(G)).$  By Schoenberg's Theorem \cite{BC}, \cite{H1}  the exponential condition is equivalent to  showing that the function $L(g) :=- \log(\c(g)) $ is negative definite, but it seems to be difficult to find conditions which ensures both negative definiteness and an extension possibility.

\section{Unbounded expectations to $\cl(\car, \a, G).$}
 \label{expect} 

The  example of Subsection \ref{Countab} shows that for any discrete  countable  group $ G$ there exist many  linear unital completely positive and $\car-$bimodular  mappings $\s$ of $B(H \otimes \ell_2(G))$ to $\cl(\car, \a, G)$ and  positive definite and strictly positive functions  $\c$ on $G$ such that  $\s(L_g\Psi(r)) = \c(g) L_g\Psi(r)$ for any $g$ in $G$ and $r$ in $\car.$ 
With this setup and the results of the previous section  we can make the following definition.
\begin{definition} Let $\car$ be a von Neumann algebra on  a Hilbert space $H$ and $G$  a countable  discrete group acting by *-automorphisms $\a_g$ on $\car. $ An unbounded expectation of $B(H \otimes \ell_2(G))$ to $\cl(\car, \a, G)$ is a pair $(\c, \s)$ such that $\s $ is a linear unital completely positive and $\car-$bimodular  mapping  of $B(H \otimes \ell_2(G))$ to $\cl( \car, \a, G)$  and $\c$ is a positive definite and strictly positive   function    such that 
\begin{itemize} 
\item[(i)]$\forall g \in G \, \forall r \in \car: \quad \s(L_g\Psi(r)) = \c(g)L_g\Psi(r).$
\item[(ii)] $ \forall x
\in B(H \otimes \ell_2(G)\, \forall g \in G: \quad \|\mathrm{Diag}\big(L_{g^{-1}}\s(x)\big)\| \leq \c(g)\|x\|.$
\end{itemize}
 \end{definition} 
Remark, that $\c(g) > 0$ for all $g$ so $\c^{-1}(g) : = 1/\c(g)$ is a positive function on $G$ and we may define a possibly unbounded Hadamard multiplier $M^{\c^{-1 }}$ on $\cl(\car,\a, G).$ Then for  $\bd : = \{x \in B(H \otimes \ell_2(G)): M^{\c^{-1}}\star_h \s(x) \text{ is bounded}\}$   the mapping 
$\Pi$ defined on $\bd$ and given by $\Pi(x) := M^{\c^{-1}} \star_h \s(x) $ is an idempotent of the ultra-weakly dense subspace $\bd $ of $B(H \otimes \ell_2(G))$  onto $\cl(\car, \a, G).$

The examples from Section  \ref{Exa} give the following theorem. 

\begin{theorem}
Let $\car$ be an injective von Neumann algebra on a Hilbert space $H$ and $\a_g$ an action of a discrete countable  group $G$ on $\car$ by *-automorphisms. 
There exists an unbounded expectation $(\c, \s)$ of $B(H \otimes \ell_2(G))$ onto $\cl(\car, \a, G).$ 
\end{theorem}

\begin{proof}
We will show the result for a unit vector $\xi = (k_g)_{(g \in G)},$ in $\ell_2(G)$ such that for each $g$ we have $k_g > 0.$ and the pair $(\c_\xi, \s_\xi),$ which was constructed in Theorem \ref{sigma}. It is easy to check that the set of unbounded conditional expectations is convex, so we may then obtain  a good collection of examples. For the pair $(\c_\xi, \s_\xi)$ most of the defining properties for an unbounded expectation follows quite easily from the construction of $\s_\xi.$ 

 The only thing which is not obvious is that the condition (ii) in the definition of an unbounded expectation is fulfilled. To see that, we will go back to the beginning of Section 3 and recall the construction of $\s_{\xi}$ from equation (\ref{sinfty}), but remember $k_g > 0$ now

\begin{equation}  \label{sinfty2} \notag
\s_\xi(x) := \sum_{t,h \in G}L_{t}\Psi\big(\a_h(\pi(x_{(th,h)})\big)
k_{th}k_h. 
\end{equation}

For each $t$ and $h$ in $G$ we have $\|L_{t}\Psi\big(\a_h(\pi(x_{(th,h)})\big)
\| \leq \|x\|$ and the sum of positive reals $\sum_{h\in G} k_{th}k_h = \langle \l_t \xi, \xi\rangle = \c_\xi(t),$ is in the interval $]0, 1], $ and the item (ii) follows. 
\end{proof} 

We find that it is worth to remark that the last lines, in the proof above, show that  for each $t$ in $G$ we may define a contraction $\Phi_t : B( H \otimes \ell^2(G)) \to \Psi(\car)$ by 
\begin{equation} \label{Fitn} \Phi_t(x) := \c_\xi^{-1}(t)\sum_{h \in G}\Psi\big(\a_h(\pi(x_{(th,h)})\big)
k_{th}k_h, \end{equation}
and that this family of contractions satisfies 
 \begin{equation} \label{Fit}\forall x \in B(H \otimes \ell_2(G)): \quad \s(x) \sim \sum_{t \in G}   \c(t)L_t \Phi_t(x).\end{equation}
 The reason why condition (ii) is included is that we get the impression that an unbounded expectations should have this form with a positive definite function which is also strictly positive, 

\subsection{The algebra $\ell_\infty(G \times G, B(H))$}
We mentioned in front of Proposition \ref{fi} that the Hadamard product in $\cl(\car, \a, G)$  directs us to look at the algebra $\ell_\infty(G, \car)$ with the point-wise product and the equation (\ref{Fit}) supports  this. If we return to equation (\ref{Fitn}) it is easy to check that the domain of definition for  $\Phi_t$ really is all of $\ell_\infty(G \times G, B(H)),$ hence we may  extend the domain of definition for  $\Phi_t$ to be a contraction of $\ell_\infty(G \times G, B(H))$ to $\Psi(\car), $ and in turn we can define a contraction mapping $\Phi$ of $\ell_\infty(G \times G, B(H))$ to $\ell_\infty(G, \car)$ by 
\begin{equation} \notag
\Phi(x) := \big(\Psi^{-1}(\Phi_t(x))\big)_{(t\in G)}.
\end{equation}
If we look at the matrix representations of the operators $L_t := \sum_g I \otimes e_{(tg,g)}$ and $\Psi(r) = \sum_g  \a_{g^{-1}}(r) \otimes e_{(g,g)}$ it is easy to see that inside the structure of $\ell_\infty (G \times G, B(H))$ we can define an isometry $\Theta$  of $\ell_\infty(G, \car)$ into $\ell_\infty ( G \times G, \car)$ \begin{align*} \label{Th}
\forall  R &= (r_g)_{(g \in G)} \in \ell_\infty(G, \car) \\ \Theta(R) :&= \sum_{g \in G} L_g \Psi(r_g) = \sum_{i,j\in G} \a_{j^{-1}}(r_{ij^{-1}})\otimes e_{(i,j)} \in \ell_\infty(G \times G, \car).
\end{align*}
We will define $\ct$ as the closed subspace of $\ell_\infty(G \times G, \car),$ which is the range of $\Theta, $ and then $\ct$ is a subalgebra of $\ell_\infty(G \times G , B(H)) $ for which there exists an isometric isomorphism of the algebra  $\ell_\infty(G, \car) $ onto this algebra. We will now show that $\ct$  is the image of an idempotent $\hat \Pi$ from $\ell_\infty( G \times G, B(H))$ of norm 1.
 From the constructions above it will follow that we may use $\hat \Pi = \Theta \circ \Phi.$ In some sense it seems natural to look at the operator algebra $\ell_\infty(G, \car) $ as the complemented subalgebra $\ct$ of $\ell_\infty(G\times G, \car).$

 We can also define $M^\c$ as a bounded linear mapping on $\ell_\infty ( G \times G, B(H))$ given by Hadamard multiplication -  the function $ x(g, h) $ is mapped to the function $ \c(gh^{-1})x(g,h)$ - and we get that the mapping $\Sigma $ defined as $\Sigma := M^\c\star_h \hat \Pi$  is a completely positive mapping  on $\ell_\infty(G \times G, B(H))$ in the following sense. We look at a  function $x(g,h)$ in $\ell_\infty( G \times G, B(H))$ as an infinite block matrix with elements in $B(H),$ and we will say that $x$ is positive if all its principal and finite sub-matrices are positive.   
 
 The comment above is somehow  suggesting itself to be part of this article, but we have no use for it, at least at this point. It should be noticed that the positivity property  of the idempotent is the {\em non obvious } part. It is quite easy  to obtain an idempotent of norm 1, because $\car$ is injective.  
\subsection{Comments}
Suppose $(\c, \s) $ is an unbounded expectation to \newline $\cl(\car, \a, G),$ then there are some  immediate questions:  

\begin{itemize} 
\item[(i)] What can you say about the domain of definition $\bd$  for the unbounded idempotent $\Pi,$  except that it contains $\cl(\car, \a, G)$ ?

\item[(ii)] For which groups $G,$ except amenable ones,  is it possible to find an unbounded expectation $(\c, \s)$ such that $\c$  is constant on conjugacy classes? 
\item[(iii)]
The construction of an unbounded expectation may be seen as an extension problem.  Suppose $\c$ is a positive definite and strictly positive  function on the group $G$ with $\c(e) = 1,$ then an unbounded expectation may be obtained if it is possible to extend the Hadamard multiplier  $M^\c$ on $\cl(\car, \a, G)$  to  a completely positive $\Psi(\car) $ bimodular mapping of $B(H \otimes \ell_2(G))$ into $\cl(\car, \a, G).$  
\item[(iv)] In Haagerup's article \cite{Ha}, he shows that for a free non abelian group with $k$ generators and any $\eps > 0$ the function  $\c_\eps(g) : = e^{-\eps \ell(g)}$ is positive definite, strictly positive with $\c_\eps(e) = 1.$  We mentioned at the end of Section \ref{Fk} that for $\eps =\frac{1}{2}\log(2k-1)$ Haagerup's $\c_\eps(g)$ function is very closely related to the positive definite function constructed in that subsection. This raises the natural question. For which $\eps > 0 $ is $M^{\c_\eps}$ part of an unbounded expectation onto $\cl(\baf_k)$?  Could it be that the lower bound for these $\eps'$s is related to the number $k$ of generators ?
   
\item[(v)] It follows from the constructions in this article that for an operator $x$ in $\cl(\car, \a, G),$ $ x \sim \sum_g L_g\Psi(r_g)$  the diagonals - $\mathrm{Diag}(L_g^*x) = \Psi(r_g)$ - play a fundamental role. We have wondered if this might leave room for an application of the positive answer \cite{MSS} to the Kadison-Singer question ?

\end{itemize} 

With respect to item (i), it is quite easy to see that any operator $ x$ in $B(H \otimes \ell_2(G))$ which is supported on a diagonal - meaning - $ x$ is of the form $x = \sum_g x_{(tg,g)} \otimes e_{(tg,g)}$ must be in $\bd,$ so it follows that $\bd$ is ultra-weakly dense in $B(H\otimes \ell_2(G)).$  

We find that the rest of the  questions  are most relevant in the situation where $\car = \bc I_H$ because the questions really deal with the structure of the positive definite functions on the group $G.$ 

Since the amenable groups give the largest possible positive definite function $\c, $  namely $\c(g) =1 $ for all $g,$  we have got the impression that the structure of the point-wise ordered set of positive and  positive definite functions associated to unbounded expectations may contain some valuable information on a given discrete countable group $G.$

With respect to item (ii) we have noticed  that for the unbounded expectation $(\c_\x, \s_\xi)$ which is constructed from a unit vector $\xi,$ we can translate $\xi$ by some $\l_t,$ then  $\l_t\xi$    gives rise to a pair $( \c_{\l_t\xi}, \s_{\l_t\xi} )$ such that $\c_{\l_t\xi} $ is conjugate to $\c_\xi.$

For item (iii), we would like to remark, that Paulsen \cite{Pa} has obtained a number of results based on extensions of certain completely positive mappings, and we hope that some of the methods used in the study of operator spaces might be applicable here.  

It is nearly obvious, that the set of  positive definite functions on $G$  with extendible multipliers  is  a convex cone, which is stable for the Hadamard, or point-wise  multiplication, product by any positive definite function. It is also stable under conjugation, so it has a structure similar to that of ideals in a C*-algebra. 

We have nothing to report on item (v), but the {\em paving version} of the Kadison-Singer question suggests to look at cut downs to diagonal block forms. 
Elements in the main diagonal of $\cl(\car, \a, G)$ have the form $\sum_g \a_{g^{-1}}(r) \otimes e_{(g,g)},$ so some  properties of a cut down to a block diagonal form may be computable.

\end{document}